\documentclass[a4paper,draft]{amsproc}
\usepackage{amssymb}
\usepackage{amscd} 
\usepackage{snapshot}
\usepackage[dvips]{graphicx}
\usepackage{color}
\usepackage{subfigure}
\usepackage{rotating}
\usepackage{amsmath}
\usepackage{amsthm}
\usepackage{amsfonts}

\theoremstyle{plain}
 \newtheorem{theorem}{Theorem}[section]
 
 \newtheorem{lemma}{Lemma}[section]
 \newtheorem{corollary}{Corollary}[section]
 \newtheorem{example}{Example}[section]
 
\theoremstyle{remark}
 
 \numberwithin{equation}{section}

\renewcommand{\leq}{\leqslant}
\renewcommand{\geq}{\geqslant}

\title[]{Locally solid topological lattice-ordered groups}



\author[Hong]{\bfseries Liang Hong}

\address{
Department of Mathematics \\ 
Robert Morris University   \\ 
Moon, PA 15108, USA}
\email{hong@rmu.edu}

\begin{document}

\vspace{18mm}
\setcounter{page}{1}
\thispagestyle{empty}

\begin{abstract}
Locally solid Riesz spaces have been widely investigated in the past several decades; but locally solid topological lattice-ordered groups seem to be largely unexplored. The paper is an attempt to initiate a relatively systematic study of locally solid topological lattice-ordered groups. We give both Roberts-Namioka-type characterization and Fremlin-type characterization of locally solid topological lattice-ordered groups. In particular, we show that a group topology on a lattice-ordered group is locally solid if and only if it is generated by a family of translation-invariant lattice pseudometrics. We also investigate (1) the basic properties of lattice group homomorphism on locally solid topological lattice-ordered groups; (2) the relationship between order-bounded subsets and topologically bounded subsets in locally solid topological lattice-ordered groups; (3) the Hausdorff completion of locally solid topological lattice-ordered groups.
\end{abstract}

\maketitle

\section{Introduction and literature review}  
Lattice-ordered groups (also called $l$-groups) are an important class of partially ordered algebraic systems. The study of lattice-ordered groups was initiated by \cite{Birkhoff1} and \cite{Clifford} and followed by many others (cf. \cite{Ball2}, \cite{Fuchs2}, \cite{Fuchs3} and \cite{Teller}). The monographs \cite{Birkhoff2} and \cite{Fuchs1} give a systematic account of the basic theory of $l$-groups. A topological lattice-ordered group (also called a topological $l$-group) is a generalization of the topological Riesz space. It can also be considered as a generalization of either a topological group or a lattice-order group. To the best of our knowledge, \cite{Smarda1} and \cite{Smarda2} first studied their basic properties and gave several fundamental results including the neighborhood theorem for topological $l$-groups. Later on, \cite{Ball1}, \cite{Gusic} and \cite{Redfield} derived some further results. Locally solid topological $l$-groups are a special class of topological $l$-groups; their relation is comparable to that between locally solid Riesz spaces and topological Riesz spaces. Recently, \cite{KR} extended the Nakano's theorem (cf. Theorem 3.3 of \cite{Nakano}) to Hausdorff topological $l$-groups. However, there seems to be no work devoted to locally solid topological $l$-groups. This paper is intended for filling this gap. In this paper, we follow the spirit of \cite{Namioka} to give a systematic investigation of basic properties of topological $l$-groups. We hope this paper will stimulate further interest along this line.

We remark that the proofs of some results in this paper might seem to be similar to their counterparts in locally solid Riesz spaces. However, a topological $l$-group has less algebraic and topological structures than a topological Riesz space; hence different theorems in both algebra and topology need to be invoked to support seemingly the same argument. Indeed, even a well-known lattice identity in Riesz spaces may no longer hold for $l$-groups. We will point out the relevant references on $l$-groups and topological groups at several places to emphasize this. We give fairly complete proofs for most results in the hope that this paper could serve as a good reference on this relatively unexplored topic.

The remainder of the paper is organized as follows. Section 2 provides readers with some basic terminologies for this paper. Section 3 gives some preliminary results of topologically $l$-groups to prepare for our main presentation; they also complement several results in \cite{Smarda1}. Section 4 studies locally solid topological $l$-groups. We give Roberts-Namioka-type characterization as well as Fremlin-type characterization of locally solid topological lattice-ordered groups; we also study basic properties of lattice group homomorphism and order-bounded subsets. Section 5 investigates topological completion of Hausdorff locally solid topological $l$-groups; in particular we extend several results in \cite{Aliprantis1} to the case of topological $l$-groups.

\section{Notation and basic concepts}
In this section, we give the basic concepts concerning Riesz spaces and lattice-ordered groups. For comprehensive monographs on these topics, we refer to \cite{Fuchs1}, \cite{LZ} and \cite{Zaanen}.

A nonempty subset $C$ of a group $G$ is called a \emph{cone} of $G$ if it satisfies the following three properties:
\begin{enumerate}
  \item [(i)]$C+C\subset C$,
  \item [(ii)]$C\cap (-C)=\{0\}$,
  \item [(iii)]$x+C-x=C$ for all $x\in G$.
\end{enumerate}
A \emph{binary relation} $\leq$ on a non-empty set $X$ is a subset of $X\times X$. A binary relation $\leq$ on a set $X$ is said to be a \emph{partial order} if it has the following three properties:
\begin{enumerate}
  \item [](Reflexivity) $x\leq x$ for $x\in X$;
  \item [](Antisymmetry) if $x \leq  y$ and $y \leq x$, then $x=y$;
  \item [](Transitivity) if $x \leq  y$ and $y \leq  z$, then $x \leq  z$.
\end{enumerate}
A set $X$ with a partial order $\leq$ is called a \emph{partially ordered set}.
A partially ordered set $X$ is called a \emph{lattice} if the infimum and supremum of any pair of elements in $X$ exist. A \emph{partially ordered group (p.o. group)} is a set $G$ satisfying the following three properties:
\begin{enumerate}
  \item [(i)]$G$ is an additive group;
  \item [(ii)]$G$ is a partially ordered set;
  \item [(iii)]$x\leq y$ implies $x+z\leq y+z$ for all $z\in G$.
\end{enumerate}
Unless otherwise stated, all groups in this paper are assumed to be commutative and written additively; the notation $\leq$ will denote the partial order of a p.o. group if no confusion may arise. An element $x$ in a p.o. group $G$ is said to be \emph{positive} or \emph{integral} if $x\geq 0$; the set of all positive elements in $G$ is called the \emph{positive cone} of $G$ and is denoted by $G_+$. A subset $C$ of a p.o. group $G$ is the positive cone of $G$ with respect to the partial order defined by $x\leq y \Longleftrightarrow y-x\in C$ if and only if $C$ is a cone of $G$. A p.o. group $G$ is said to be \emph{Archimedean} if $nx\leq y$ for $x, y\in G$ and all $n\in N$ implies $x=0$. A p.o. group is called a \emph{lattice-ordered group ($l$-group)} if it is a lattice at the same time. A subgroup of an $l$-group is called an \emph{lattice-ordered subgroup ($l$-subgroup)} if it is a lattice. For two elements $x$ and $y$ in an $l$-group, $x\vee y$ , $x\wedge y$ denotes $\sup\{x, y\}$ and $\inf \{x, y\}$, respectively; we also define $x^+=x\vee 0, x^-=(-x)\vee 0$ and $|x|=x\vee (-x)$. If $a$ and $b$ are two elements in an $l$-group, then the set $[a, b]=\{x\mid a\leq x\leq b\}$ is called an \emph{ordered interval}. A subset $E$ of $G$ is said to be \emph{order-bounded} $E$ is contained in some ordered interval. A subset $E$ of $G$ is said to be \emph{solid} if $|x|\leq |y|$ and $y\in E$ implies $x\in E$. Every subset $E$ of $G$ is contained in the solid set $Sol(E)=\{x\in G\mid |x|\leq |y|\text{ for some $y\in E$}\}$; we call $Sol(E)$ the \emph{solid hull} of $E$. An $l$-subgroup $H$ of an $l$-group $G$ is said to be \emph{order dense} in $G$ if for every $0<x\in G$ there exists an element $y\in H$ such that $0<y\leq x$.

Let $G$ be an $l$-group. A net $(X_{\alpha})_{\alpha\in A}$ is said to be \emph{decreasing} if $\alpha\geq \beta$ implies $x_{\alpha}\leq x_{\beta}$. The notation $x_{\alpha}\downarrow x$ means $(x_{\alpha})_{\alpha\in A}$ is a decreasing net and the infimum of the set $\{x_{\alpha}\mid \alpha\in A\}$ is $x$. A net $(x_{\alpha})_{\alpha\in A}$ in an $l$-group $G$ is said to be \emph{order convergent} to an element $x\in G$, written as $x_{\alpha}\xrightarrow{o} x$, if there exists another net $(y_{\alpha})_{\alpha\in A}$ in $G$ such that $|x_{\alpha}-x|\leq y_{\alpha}\downarrow 0$; if a topological $\tau$ is also present, we will use $x_{\alpha}\xrightarrow{\tau} x$ to denote topological convergence. A solid subgroup of an $l$-group $G$ is called an \emph{ideal}; a $\sigma$-order closed ideal of $G$ is called a \emph{$\sigma$-ideal}; an order-closed ideal of $G$ is called a \emph{band}.

A \emph{topological lattice-ordered group (topological $l$-group)} $(G, \tau)$ is a topological space such that
\begin{enumerate}
  \item [(i)]$G$ is an $l$-group;
  \item [(ii)]the group and lattice operations are all continuous, that is, the following four operations are continuous:
                \begin{enumerate}
                  \item [(1)](Continuity of Addition) the map $(x, y)\mapsto x+y$, from $G\times G$ to $G$, is continuous;
                  \item [(2)](Continuity of Inverse): the map $x\mapsto -x$, from $G$ to $G$, is continuous;
                  \item [(3)](Continuity of Join): the map $(x, y)\mapsto x\vee y$, from $G\times G$ to $G$, is continuous;
                  \item [(4)](Continuity of Meet): the map $(x, y)\mapsto x\wedge y$, from $G\times G$ to $G$, is continuous.
                \end{enumerate}
\end{enumerate}
\noindent \textbf{Remark} It is well-known that an $l$-group and a topological group both can be defined in several different but equivalent ways (cf. Chapter 1 of \cite{AT} and Chapter V of \cite{Fuchs1}); it follows that the above definition of topological $l$-groups also has quite a few equivalent definitions. For example, we can say a topological space $G$ is \emph{a topological $l$-group} if
\begin{enumerate}
  \item [(i)]$G$ is an $l$-group;
  \item [(ii)]the group and lattice operations are all continuous, that is, the following four operations are continuous:
                \begin{enumerate}
                  \item [(1)](Continuity of Subtraction) the map $(x, y)\mapsto x-y$ from $G\times G$ to $G$ is continuous;
                  \item [(2)](Continuity of Join): the map $(x, y)\mapsto x\vee y$ from $G\times G$ to $G$ is continuous;
                  \item [(3)](Continuity of Meet): the map $(x, y)\mapsto x\wedge y$ from $G\times G$ to $G$ is continuous.
                \end{enumerate}
  \end{enumerate}
Henceforth, $\mathcal{N}_x$ will denote the neighborhood system at a point $x$; $\mathcal{B}_x$ will denote a neighborhood base at $x$. When no confusion may result, we often write $(G, \tau)$ as $G$; when we need to emphasize or refer to the topology $\tau$ on $G$, we often use the full notation $(G, \tau)$. Different from some authors, such as \cite{Willard}, we do not assume a topological group is Hausdorff.

Let $(G, \tau)$ be a topological $l$-group. The group topology $\tau$ is said to be \emph{locally solid} if $\tau$ has a neighborhood base at zero consisting of solid sets; in this case $(G, \tau)$ is said to be a \emph{locally solid topological $l$-group}.

Let $T$ be a group homomorphism between two topological $l$-groups $(G_1, \tau_1)$ and $(G_2, \tau_2)$.
$T$ is said to be a \emph{positive homomorphism} if carries positive elements to positive elements; it is said to be a \emph{lattice homomorphism} if $(x\vee y)=T(x)\vee T(y)$ for all $x, y\in G$; it is said to be an \emph{order-bounded} if it carries order-bounded sets to order-bounded sets; it is said to be   \emph{topologically continuous} if $T^{-1}(O)\in \tau_1$ for every open set $O\in \tau_2$; it is said to be \emph{$\sigma$-order-continuous} if sequence $(T(x_{n}))$ is order-convergent for every order-convergent sequence $(x_n)$ in $G_1$;  it is said to be \emph{order-continuous} if the net $(T(x_{\alpha}))$ is order-convergent for every order-convergent net $(x_{\alpha})$ in $G_1$.

$R_+$ will denote the set of all nonnegative reals, that is, $R_+=\{a\mid a\in R\ \text{and}\ a\geq 0\}$.
A \emph{pseudometric} on a set $X$ is a  mapping $d: X\times X\rightarrow R_+$ such that for all $x, y, z\in X$:
\begin{enumerate}
  \item [(i)]$d(x, y)=d(y, x)$;
  \item [(ii)]$d(x, y)\leq d(x, z)+d(y, z)$.
\end{enumerate}
A pseudometric on a set $X$ is said to be \emph{translation-invariant} if $d(x, y)=d(x+z, y+z)$ for all $x, y, z\in X$. A pseudometric on an $l$-group $G$ is said to be a \emph{lattice pseudometric} if $d(0, x)\leq d(0, y)$ whenever $x\leq y$ in $G$.

\section{Some preliminary results of topological $l$-groups}
A topological $l$-group is a topological group; hence it inherits all properties of a topological group. In particular, we have the following theorem (cf. Chapter III of \cite{Husain}).
\begin{theorem}\label{theorem3.1}
Let $G$ be a topological $l$-group. Then the following statements hold.
\begin{enumerate}
  \item [(i)]$G$ is regular.
  \item [(ii)]$G$ is homogeneous, that is, for any two given points $x, y\in G$, there exists a homeomorphism $f$ of $G$ onto $G$ such that $f(x)=y$.
  \item [(iii)]If $H$ is subgroup of $G$, then $\overline{H}$ is also a subgroup of $G$.
  \item [(iv)]If $H$ is subgroup of $G$ and $H$ is open, then $H$ is closed.
  \item [(v)]If $H$ is subgroup of $G$, then $H$ is discrete if and only if $H$ has an isolated point.
  \item [(vi)]If $H$ is subgroup of $G$ and $H$ is open, then the interior of $H$ is nonempty.
\end{enumerate}
\end{theorem}
\noindent \textbf{Remark.} The above observation simplifies several proofs in \cite{Smarda1} (e.g. Theorem 1.4 (1) (3), Theorem 3.1, Corollary of Theorem 3.1). On the other hand, the structure of a topological $l$-group is richer than that of a topological group; hence we would expect some stronger results. This will be clear from our further discussion in this section.

\begin{theorem}\label{theorem3.2}
Let $G$ be a topological $l$-group and $\mathcal{B}_0$ be the neighborhood base at $0$. Then the following statements hold.
\begin{enumerate}
  \item [(i)]The operation $x\mapsto |x|$, from $G$ to $G$, is continuous.
  \item [(ii)]$x+\mathcal{B}_0=\{x+B\mid B\in\mathcal{B}_0\}$ is a neighborhood base for $\mathcal{N}_x$.
  \item [(iii)]For any neighborhood $U$ of zero, there exists another neighborhood $V$ of zero such that
               $V^+=\{x^+\mid x\in V\}\subset U,\ V^-=\{x^-\mid x\in V\}\subset U$, and
               $|V|=\{|x|\mid x\in V\}\subset U$.
  \item[(iv)]If $K$ is a compact set contained in an open set $O$, then there exists a neighborhood $U$ of zero such that $K+U\subset O$.
  \item[(v)]The sum of two open sets is open.
  \item[(vi)]The sum of compact set and a closed set is closed.
  \item[(vii)]If $E_1$ and $E_2$ are two subsets of $G$, then $\overline{E_1}+\overline{E_2}=\overline{E_1+E_2}$.
\end{enumerate}
\end{theorem}

\begin{proof}
Only (i) and (iii) needs a proof; the remaining statements hold for a topological group (cf. p. 54 of \cite{Pontrjagin}); hence they hold for a topological $l$-group.
\begin{enumerate}
  \item [(i)]$G$ is a topological $l$-group; hence the maps $x\mapsto x$ and $x\mapsto -x$ are continuous. Since $G\times G$ is understood to carry the product topology, $x\mapsto (x, -x)$ is continuous. In view of the continuity of $(x, y)\mapsto x\vee y$; the composition $x\mapsto |x|=x\vee(-x)$ is continuous too.
  \item [(iii)]The conclusion follows from the continuity of the maps $x\mapsto x^+, x\mapsto x^-$ and $x\mapsto |x|$.
\end{enumerate}
\end{proof}
\noindent \textbf{Remark.}  In general, $|x|+\mathcal{B}_0=\{|x|+B\mid B\in\mathcal{B}_0\}$ is not a neighborhood base for $\mathcal{N}_x$, because the map $x\mapsto |x|$ may not have an inverse. Consider the following example.

\begin{example}\label{example3.1}
Let $G$ be the additive group on $R$ equipped with the usual topology and the usual order. Then $G$ is evidently a topological $l$-group. Take $x=-1$. Then $|x|+\mathcal{B}_0$ is the neighborhood base at $1$ which is evidently not a neighborhood base at $-1$.
\end{example}

\begin{theorem}[Separation property]\label{theorem3.3}
Let $(G, \tau)$ be a topological $l$-group and $\mathcal{N}_0$ be its $\tau$-neighborhood system at zero. Then the following statements are equivalent.
\begin{enumerate}
  \item [(i)]$G$ is a $T_0$-space.
  \item [(ii)]$G$ is a Hausdorff space.
  \item [(iii)]$\cap_{U\in \mathcal{N}_0} U=\{0\}$.
  \item [(iv)]$\forall x\in G\backslash \{0\}$, there exists a neighborhood $U$ of zero such that $x\not\in U$.
  \item [(v)]$\forall x\in G\backslash \{0\}$, there exists a neighborhood $U$ of zero such that $x^+\not\in U$.
  \item [(vi)]$\forall x\in G\backslash \{0\}$, there exists a neighborhood $U$ of zero such that $x^-\not\in U$.
  \item [(vii)]$\forall x\in G\backslash \{0\}$, there exists a neighborhood $U$ of zero such that $|x|\not\in U$.
\end{enumerate}
\end{theorem}

\begin{proof}
The equivalence of (i)-(iv) holds for a topological group (cf. p. 48 of \cite{Husain}); therefore it holds for a topological $l$-group. Take any element $x\in G$. Since $G$ is a lattice, $x^+, x^-$ and $|x|$ are all elements in $G$. Therefore, the equivalence of (v), (vi) and (vii) follow from the equivalence of (i) and (iv).
\end{proof}

It is well-known that a linear operator between two normed spaces is continuous if it is continuous at one point; likewise, a homomorphism between two topological groups is continuous if it is continuous at one point. For a group homomorphism between two topological $l$-groups, the following result is obvious.

\begin{theorem}\label{theorem3.4}
Let $T$ be a homomorphism between two topological $l$-groups $G_1$ and $G_2$. If $T$ is continuous at $x^+_0$ for a point $x_0\in G_1$, then $T$ is uniformly continuous. Similarly, if $T$ is continuous at $x^-_0$ for a point $x_0\in G_1$, then $T$ is uniformly continuous.
\end{theorem}

We conclude this section by recalling the characterization theorem of a topological $l$-group in terms of the neighborhood base at zero (cf. Theorem 1.2 of \cite{Smarda1}); this result will be needed in the next section.

\begin{theorem}\label{theorem3.5}
Let $(G, \tau)$ be a topological $l$-group and $\mathcal{B}_0$ be a neighborhood base at zero. Then $\mathcal{B}$ satisfies the following conditions.
\begin{enumerate}
  \item [(i)]If $U\in \mathcal{B}_0$, then there exists $V\in\mathcal{B}_0$ such that $V+V\subset U$.
  \item [(ii)]If $U\in \mathcal{B}_0$, then $-U\in \mathcal{B}_0$.
  \item [(iii)]If $U\in\mathcal{B}_0$ and $x\in U$, then there exists $V\in\mathcal{B}_0$ such that $x+V\in U$.
  \item [(iv)]If $U\in\mathcal{B}_0$ and $x\in G$, then there exists $V\in\mathcal{B}_0$ such that $(V-x^+)\vee(V+x^-\subset U$.
\end{enumerate}
Conversely, if a filter $\mathcal{F}$ of subsets of an $l$-ordered $G$ satisfies properties (i)-(iv), then $\mathcal{F}$ uniquely determines a lattice group topology on $G$.
\end{theorem}

\section{Locally solid topological $l$-groups}
The class of locally solid Riesz spaces is a special class of ordered topological vector spaces; it has been extensively studied in the past several decades (cf. \cite{AB1} and the references listed there). However, locally solid topological $l$-groups, as a special class of topological $l$-groups, are almost unexplored. To the best of our knowledge, only \cite{KR} generalized the Nakano's theorem from Hausdorff locally solid Riesz spaces to Hausdorff locally solid topological $l$-groups. In this section, we try to systematically describe the basic properties of locally solid topological $l$-group in the same spirit of \cite{Namioka}. In our presentation, we will need the following basic result a few times.

\begin{lemma}\label{lemma4.1}
If $G$ is an $l$-group and $x, y, z\in G$, then the following identities hold.
\begin{enumerate}
  \item [(i)]$x+(y\vee z)=(x+y)\vee (x+z)$.
  \item [(ii)]$x+(y\wedge z)=(x+y)\wedge (x+z)$.
  \item [(iii)]$x\vee y=(y-x)^+ +x=(x-y)^+ +y$.
  \item [(iv)]$x\wedge y=x-(x-y)^+$.
  \item [(v)]$x+y=x\vee y +x\wedge y$.
  \item [(vi)]$x=x^- -x^-$.
  \item [(vii)]$|x|=x^+ + x^-$.
  \item [(viii)]$x\wedge y=-[(-x)\vee(-y)]$.
\end{enumerate}
\end{lemma}
\begin{proof}
See \cite{Birkhoff1} and \cite{Fuchs1}.
\end{proof}

First, we give a characterization theorem for locally solid group topologies on $l$-groups; the result is an extension of the Roberts-Namioka characterization theorem for locally solid linear topologies on Riesz spaces (cf. \cite{Namioka} and \cite{Roberts}).

\begin{theorem}\label{theorem4.1}
Let $(G, \tau)$ be a topological $l$-group. Then the following statements are equivalent.
\begin{enumerate}
  \item [(i)]$(G, \tau)$ is a locally solid topological $l$-group.
  \item [(ii)]The map $(x, y)\mapsto x\vee y$, from $G\times G$ to $G$, is uniformly continuous.
  \item [(iii)]The map $(x, y)\mapsto x\wedge y$, from $G\times G$ to $G$, is uniformly continuous.
  \item [(iv)]The map $x\mapsto x^-$, from $G$ to $G$, is uniformly continuous.
  \item [(v)]The map $x\mapsto x^+$, from $G$ to $G$, is uniformly continuous.
\end{enumerate}
\end{theorem}

\begin{proof}
\begin{enumerate}
  \item [](i) $\Longrightarrow$ (ii). By Birkhoff's inequality (cf. Equation (27) in \cite{Birkhoff1}), we have
         \begin{equation*}
         |x\vee y-w\vee z|\leq |x- w|+|y- z|.
         \end{equation*}
         By hypothesis, we may choose a solid neighborhood $V$ of zero. If $x-w\in V$ and $y-z\in V$, then
         $|x-w|+|y- z|\in V$ 
         by Theorem \ref{theorem3.5}. It follows from the solidness of $V$ that $x\vee w-y\vee z\in V$, proving that the map $(x, y)\mapsto x\vee y$ is uniformly continuous.
  \item [](ii) $\Longrightarrow$ (iii). Since $x\wedge y=-[(-x)\vee (-y)]$ holds in a topological $l$-group,
      the conclusion follows.
  \item [](iii) $\Longrightarrow$ (iv). The conclusion follows from the identity $x^- =-(x\wedge 0)$.
  \item [](iv) $\Longrightarrow$ (v). This follows from the identity $x^+=(-x)^-$.
  \item [](v) $\Longrightarrow$ (i). Let $U$ be a neighborhood at zero. We need to find a solid neighborhood that is contained in $U$. By Theorem \ref{theorem3.5}, we can choose a symmetric neighborhood $U'$ at zero such that $U'+U'\subset U$. 
          Since the map $x\mapsto x^+$ is uniformly continuous, we can choose a symmetric neighborhood $V$ at zero such that $x-y\in V$ implies $x^+-y^+\in U'$. Next, choose a symmetric neighborhood $W$ at zero such that $W+W\subset V$; then apply the uniform continuity of the map $x\mapsto x^+$ again to choose a symmetric neighborhood $W'$ at zero such that $x-y\in W$ implies $x^+-y^+\in W$. To complete the proof, we show that the solid hull $Sol(W')$ of $W'$ is a subset of $U$. To this end, assume $|x|\leq |y|$ and $y\in W'$. By our choice of $W$, we have $y^+\in W$ and $y^- \in W$; hence $x^+ - (|y|-x^+)=|y|=y^+ + y^-\in W+W\subset V$, implying $x^+=x^+-(|y|-x^+)^+\in U'$. Similarly, we have $x^-\in U'$. Therefore, $x=x^+-x^-\in U'+U'\subset U$, proving $Sol(W')\subset U$.

\end{enumerate}
\end{proof}
\noindent \textbf{Remark 1.} By definition of a topological $l$-group, the maps $(x, y)\mapsto x\vee y$ and $(x, y)\mapsto x\wedge y$ are both continuous; however, if $(G, \tau)$ is no locally solid, then there is no guarantee that it is uniformly continuous. Example 2.18 of \cite{AB1} may be used to illustrate this point.\\

\noindent \textbf{Remark 2.} If $(G, \tau)$ is locally solid, then the map $x\mapsto |x|$, from $G\times G$ to $G$, is uniformly continuous (by (iii) and the fact $|x|=-[(-x)\wedge x]$); but the converse is not true. To see this, consider the following example.

\begin{example}\label{example4.1}
Let $G$ the group of $R^2$ under the usual pointwise addition. Equip $G$ with the usual topology $\tau_{u}$ and the lexicographic order. Then $(G, \tau_{u})$ is obviously a topological $l$-group. It is clear that the map $x\mapsto |x|$ is uniformly continuous. However, $\tau_{u}$ is not locally solid. Otherwise, any order-bounded interval would be $\tau_{u}$-bounded. But this is not the case. To see this, consider the order-bounded interval $[x, y]$, where $x=(0, 0)$ and $y=(1, 0)$. Since $[x, y]$ contains vertical infinite rays, it cannot be be the $\tau_{u}$-bounded.
\end{example}

It is well-known that a linear topology on a vector space is locally convex if and only if it is generated by a family of seminorms (cf. p. II.24 of \cite{Bourbaki}). Fremlin proved a similar result for linear topologies on Riesz spaces: a linear topology on a Riesz space is locally solid if and only if it is generated by a family of Riesz pseudonorms (cf. 22C of \cite{Fremlin}). Below we show that a group topology on an $l$-group is locally solid if and only if it is generate by a family of invariant lattice pseudometrics.

\begin{theorem}
A group topology $\tau$ on an $l$-group $G$ is locally solid if and only if it is generated by a family of translation-invariant lattice pseudometrics.
\end{theorem}

\begin{proof}
Suppose $\{d_{\alpha}\}_{\alpha\in A}$ is a family of translation-invariant lattice pseudometrics. Let $d$ be an arbitrary pseudometric in this family. For every $r>0$, put
\begin{equation*}
B_d(0, r)=\{x\in G\mid d(0, x)<r\}.
\end{equation*}
Then the translation-invariant of $d$ implies $B_d(0, r)$ is symmetric, i.e., $B_d(0, r)=-B_d(0, r)$; the subadditivity of $d$ implies $B_d(0, \frac{r}{2})+B_d(0, \frac{r}{2})\subset B_d(0, r)$. Next, assume $|x|\leq |y|$ in $G$ and $y\in B_d(0, r)$. Since $d$ is a lattice pseudometric, we have $d(0, x)\leq d(0, y)<r$, showing that $B_d(0, x)$ is solid subset of $G$. Thus, for any finitely many $d_1, ..., d_n$ in $\{d_{\alpha}\}_{\alpha\in A}$, the collection of all sets of the form
\begin{equation*}
B_{d_1}(0, r)\cap...\cap B_{d_n}(0, r), \quad r>0,
\end{equation*}
is a neighborhood base at zero for some locally solid group topology on $G$. It follows that the family $\{d_{\alpha}\}_{\alpha\in A}$ generates a locally solid group topology on $G$.

Conversely, suppose $\tau$ is a translation-invariant locally solid group topology on an $l$-group $G$, we need to show that $\tau$ is generated by a family of translation-invariant lattice pseudometrics. To this end, let $V$ be a neighborhood at zero. Choose a sequence $\{U_n\}$ of locally solid symmetric $\tau$-neighborhoods of zero such that
\begin{eqnarray*}
& & U_1 = V; \\
& & U_{n+1}+U_{n+1}+U_{n+1}\subset U_n, \quad \forall n\geq 1.
\end{eqnarray*}
Define a function $\rho: G\times G\rightarrow R_+$ as follows:
\begin{equation}\label{4.1}
\rho(x, y)=\left\{
             \begin{array}{ll}
               1, & \hbox{if $x-y\not \in U_1$;} \\
               2^{-n}, & \hbox{if $x-y\in U_{n+1}\backslash U_n$;} \\
               0, & \hbox{if $x-y\in \cap_{n=1}^{\infty} U_n$.}
             \end{array}
           \right.
\end{equation}
Then $\rho$ has the following three properties.
\begin{enumerate}
  \item [(i)]$\rho$ is translation-invariant, although it is not a pseudometric.
  \item [(ii)]$x-y\in U_n$ if and only if $\rho(x, y)\leq 2^{-n}$ for $x, y\in G$.
  \item [(iii)]$\rho(0, x)\leq \rho(0, y)$ whenever $|x|\leq |y|$ and $x, y\in G$. 
\end{enumerate}
Next, we define a function $d: G\times G\rightarrow R_+$ via the formula
\begin{equation}\label{4.2}
d(x, y)=\inf\left\{\sum_{i=1}^{n-1}\rho(x_i, x_{i+1})\ \bigg |\  \normalsize x_1=x, x_n=y, x_i\in G \text{ for $i=2, ..., n-1$}\right\}.
\end{equation}
We claim that $d$ is a translation-invariant pseudometric on $G$. Indeed, it is evident that $d(x, y)\geq 0$ and $d(x, y)=d(y, x)$. It is also easy to see from Equation (\ref{4.1}) and Equation (\ref{4.2}) that $d(x, y)\leq d(x, z)+d(z, y)$ for all $x, y, z\in G$. Since $\rho$ is translation-invariant, Equation (\ref{4.2}) shows that $d$ is translation-invariant too. Finally, suppose $x, y\in G$ and $y=\sum_{i=1}^n y_i$, where $y_1, ..., y_n\in G$. Then the dominated decomposition property of $l$-groups (cf. p. 69 of \cite{Fuchs1}) implies the existence of $x_1, ..., x_n\in G$ such that $x=\sum_{i=1}^n x_i$ and $|x_i|\leq |y_i|$ for $i=1, .., n$. It follows from property (iii) of $\rho$ that
\begin{equation*}
d(0, x)\leq \sum_{i=1}^{n-1} \rho(0, x_0)\leq \sum_{i=1}^{n-1}\rho(0, x_i),
\end{equation*}
implying $d(0, x)\leq d(0, y)$. Therefore, $d$ is a translation-invariant lattice pseudometric on $G$.

The above discussion shows that for each neighborhood $V$ of zero, there exists a translation-invariant pseudometric $d_V$ on $G$ such that
\begin{equation}\label{4.3}
x\in V \text{ if and only if $d_V(0, x)\leq 1$}.
\end{equation}
Let $\tau'$ be the group topology generated by $\{d_V\}_{V\in\mathcal{N}_0}$. Then Equation (\ref{4.3}) implies that $\tau\subset \tau'$. To finish the proof, we need to show $\tau'\subset \tau$. To this end, it suffices to show that for any positive integer $n$ we have
\begin{equation}\label{4.4}
B(0,2^{-n})=\{x\in G\mid d(0, x)<2^{-n}\}\subset U_n.
\end{equation}
It is easy to see that Equation (\ref{4.4}) is implied by $\rho\leq 2d$ which is further implied by
\begin{equation}\label{4.5}
\frac{1}{2}\rho(x, y)\leq \sum_{i=1}^{n-1} \rho(x_i, x_{i+1}),
\end{equation}
where $x_1=x, x_n=y$ and $x_2, ..., x_{n-1}\in G$. If $\sum_{i=1}^{n-1}\rho(x_i, x_{i+1})=0$, then Equation (\ref{4.1}) and Theorem \ref{theorem3.5} imply that
$\rho(x, y)=0$; hence Equation (\ref{4.5}) holds. For the remainder of the proof, we assume that $\sum_{i=1}^{n-1}\rho(x_i, x_{i+1})\neq 0$. We establish Equation (\ref{4.5}) by induction on $n$. The case $n=1$ is trivial. For the inductive step, we assume Equation (\ref{4.5}) holds for all positive integers that are less than $n$. Consider two cases.

Case I: $\sum_{i=1}^{n-1} \rho(x_i, x_{i+1})<\frac{1}{2}$. If $\sum_{i=1}^{n-1} \rho(x_i, x_{i+1})=0$, then we clearly have $x_i-x_{i+1}\in U_n$ for all $n\in N$; hence $x-y\in \cap_{n=1}^\infty U_n$ implying $\rho(x, y)=0$. 
Next, we assume $\sum_{i=1}^{n-1} \rho(x_i, x_{i+1})>0$. Put
\begin{equation*}
m=\max_{1\leq j\leq n}\left\{j\ \bigg | \  \frac{1}{2}\sum_{i=1}^n \rho(x_i, x_{i+1})\geq \sum_{i=1}^j\rho(x_i, x_{i+1}) \right\}.
\end{equation*}
Then $\frac{1}{2}\sum_{i=1}^{n-1} \rho(x_i, x_{i+1})<\sum_{i=1}^{m+1} \rho(x_i, x_{i+1})$ which leads to
$\sum_{i=m+1}^{n-1} \rho(x_i, x_{i+1})<\frac{1}{2}\sum_{i=1}^{n-1} \rho(x_i, x_{i+1})$. By the induction hypothesis, $ \frac{1}{2} \rho(x, x_m)\leq \sum_{i=1}^{m-1} \rho(x_i, x_{i+1})$; hence $\rho(x, x_m)\leq \sum_{i=1}^{n-1} \rho(x_i, x_{i+1})$.
Likewise, we have
\begin{equation*}
f(x_{m+1}, y)\leq \sum_{i=1}^{n-1} \rho(x_i, x_{i+1}).
\end{equation*}
Put
\begin{equation*}
j=\min_{ k\geq 1}\left\{ k\ \bigg |\ 2^{k-1}\leq \sum_{i=1}^{n-1} \rho(x_i, x_{i+1}) \right\}.
\end{equation*}
Then $\rho(x, x_m)<2^{j-1}$, implying $x-x_m\in U_{j-1}$. Similarly, we have $x_m-x_{m+1}\in U_{j-1}$ and
$x_{m+1}-y\in U_{j-1}$. By the choice of $\{U_n\}$, we have $x-y\in U_j$. Therefore, property (ii)  of $\rho$ implies that $\frac{1}{2}\rho(x, y)\leq 2^{-\j}\leq \sum_{i=1}^{n-1} \rho(x_i, x_{i+1})$, that is, (\ref{4.5}) holds.

Case II: $\sum_{i=1}^{n-1} \rho(x_i, x_{i+1})\geq \frac{1}{2}$. In this case, (\ref{4.5}) holds trivially in view of (\ref{4.1}).
\end{proof}

Theorem \ref{theorem3.3} shows that the set $A=\cap_{U\in \mathcal{N}_0} U$ in a topological $l$-group $(G, \tau)$ plays an important role in characterizing the separation property of $\tau$. From Theorem \ref{theorem3.5} we see that $A$ is always $\tau$-closed.
When $\tau$ is locally solid, we can say more.

\begin{theorem}
If $(G, \tau)$ is a locally solid topological $l$-group and $\mathcal{N}_0$ is the $\tau$-neighborhood system at zero, then the set $A=\cap_{U\in\mathcal{N}_0} U$ is a $\tau$-closed ideal of $G$.
\end{theorem}

\begin{proof}
Let $U$ be an arbitrary $\tau$-neighborhood at zero. Since $\tau$ is locally solid, $U$ contains a $\tau$-closed solid $\tau$-neighborhood of zero. It follows that $A$ is a solid subset of $G$. Next, take $x, y\in A$ and
choose a $\tau$-neighborhood symmetric $V$ of zero such that $V+V\subset U$.
Then $x-y\in V+V\subset U$, implying $x-y\in A$. Since $A$ is evidently nonempty, this shows that $A$ is subgroup of $G$. Therefore, $A$ is a $\tau$-closed ideal of $G$.
\end{proof}

\begin{theorem}
Suppose $(G, \tau)$ is a locally solid topological $l$-group and $G$ is an order dense subset of an $l$-group $H$. If $\tau$ extends to a locally solid lattice group topology $\tau_H$ on $H$, then $(G, \tau_H)$ is a Hausdorff locally solid topological $l$-group.
\end{theorem}

\begin{proof}
Take any $x\in H$. Without loss of generality, we may assume $x>0$. Since $G$ is order dense in $H$, we can choose a $y\in G$ such that $0<y\leq x$.
As $\tau$ is a Hausdorff group topology, we can pick a $\tau$-neighborhood $U$ of zero such that $y\not \in U$.
Next, choose a solid $\tau_H$-neighborhood $V$ of zero such that $G\cap V \subset U$.
In view of Theorem \ref{theorem3.3}, it remains to show $x\not \in V$. We proceed by contraposition. If $x\in V$, then $y\in V$ by the solidness of $U$; hence $y\in G\cap V \subset U$, contradicting our choice of $U$. Therefore, $x\not\in V$.
\end{proof}

\cite{Goffman}, \cite{Jaffard} and \cite{Pierce} gave some properties of lattice homomorphisms between $l$-groups. The next two theorems extend two characterization theorems of lattice homomorphisms between Riesz spaces (cf. Theorem 2.14 and Theorem 2.21 of \cite{AB2}) to the case of $l$-groups.

\begin{theorem}\label{theorem4.2}
Let $T$ be a group homomorphism between two $l$-groups $G_1$ and $G_2$. The the following statements are equivalent.
\begin{enumerate}
  \item [(i)]$T$ is a lattice homomorphism.
  \item [(ii)]$T(x^+)=(T(x))^+$ for all $x\in G_1$.
  \item [(iii)]$T(x\wedge y)=T(x)\wedge T(y)$ for all $x, y\in G_1$.
  \item [(iv)]$T(x)\wedge T(y)=0$ whenever $x\wedge y=0$ in $G_1$.
  \item [(v)]$T(|x|)=|T(x)|$ for all $x\in G_1$.
\end{enumerate}
\end{theorem}

\begin{proof}
\begin{enumerate}
  \item [] (i) $\Longrightarrow$ (ii). Let $T$ is a lattice homomorphism and $x\in G_1$. Then
            \begin{equation*}
            T(x^+)=T(x\vee 0)=T(x)\vee T(0)=T(x)\vee 0=(T(x))^+.
            \end{equation*}
  \item [] (ii) $\Longrightarrow$ (iii). Take two points $x, y\in G_1$. In view of Lemma \ref{lemma4.1} (iv), statement (ii) implies
            \begin{eqnarray*}
            T(x\wedge y) &=& T(x-(x-y)^+) \\
                         &=& T(x)-T((x-y)^+)\\
                         &=& T(x)-(T(x-y))^+\\
                         &=& T(x)-(T(x)-T(y))^+=T(x)\wedge T(y).
            \end{eqnarray*}
  \item [](iii) $\Longrightarrow$ (iv). If $x\wedge y=0$ in $G_1$, then (iii) implies
            \begin{equation*}
            T(x)\wedge T(y) =T(x\wedge y)=T(0)=0.
            \end{equation*}
  \item [](iv) $\Longrightarrow$ (v).
  Let $x\in G_1$. Then Lemma \ref{lemma4.1} (v) shows
            \begin{equation*}
            |T(x^+)-T(x^-)|=T(x^+)\vee T(x^-)-T(x^+)\wedge T(x^-).
            \end{equation*}
            Since $x^+\wedge x^-=0$, (iv) and the fact that $T$ is a lattice homomorphism imply
            \begin{eqnarray*}
            |T(x)| &=& |T(x^+)-T(x^-)|\\
                    &=& T(x^+)\vee T(x^-) =T(x^+ \vee x^-)\\
                    &=& T(x^++x^-)=T(|x|).
            \end{eqnarray*}
  \item [](v) $\Longrightarrow$ (i). Take two elements $x, y\in G_1$. Apply Lemma \ref{lemma4.1} to get
            \begin{eqnarray*}
            x+y+|x-y| &=& (x+y)+(x-y)\vee [-(x-y)]\\
                       &=& (2x)\vee (2y)\\
                       &=& 2 (x\vee y). 
            \end{eqnarray*}
            Therefore, (v) implies
            \begin{eqnarray*}
            2T(x\vee y)=T(2 (x\vee y)) &=& T(x+y+|x-y|) \\
                            &=& T(x)+T(y)+T(|x-y|)\\ 
                            &=& T(x)+T(y)-|T(x)-T(y)|\\
                            &=& 2 [T(x)\vee T(y)].
            \end{eqnarray*}
            Since an element in an $l$-group has an infinite order (Alternatively, recall that we assume that all $l$-groups are commutative; hence the cancellation law holds.),
            it follows that $T(x\vee y)=T(x)\vee T(y)$, that is, $T$ is a lattice homomorphism.
\end{enumerate}
\end{proof}

\begin{theorem}\label{theorem4.3}
Let $T$ be a lattice homomorphism between two $l$-groups $G_1$ and $G_2$. Then the following statements hold.
\begin{enumerate}
  \item [(i)]$T$ is positive.
  \item [(ii)]$T(G_1)$ is a topological $l$-group.
  \item [(iii)]If $T$ is order-continuous, then $T$ preserves all suprema and infima of a nonempty subset in $G_1$.
  \item [(iv)]If $T$ is onto, then $T$ maps solid sets in $G_1$ to solid sets in $G_2$.
  \item [(v)]If $T$ is bijective, then $T$ and $T^{-1}$ are both positive.
  \item [(vi)] The kernel $Ker(T)$ of $T$ is an ideal of $G_1$.
  \item [(vii)]If $T$ is onto, then $T$ is $\sigma$-order-continuous if and only if $Ker(T)$ is a $\sigma$-ideal of $G_1$.
  \item [(viii)]If $T$ is onto, then $T$ is order-continuous if and only if $Ker(T)$ is a band of $G_1$.
\end{enumerate}
\end{theorem}

\begin{proof}
\begin{enumerate}
  \item [(i)]Theorem \ref{theorem4.2} shows that  $T(x)\geq 0$ for $x\in (G_1)_+$; hence $T$ is positive.
  \item [(ii)]This follows immediately from the definition of lattice homomorphisms.
  \item [(iii)]This is evident.
  \item [(iv)]Let $E$ be a solid subset of $G_1$. Suppose $ |w|\leq |z|, z\in T(E)$ and $w\in G_2$. Then there exist $x\in G_1$ and $y\in E$ such that $w=T(x)$ and $z=T(y)$.
              Since $|T(x)|\leq |T(y)|$, Theorem \ref{theorem4.2} implies
              \begin{equation*}
              T(x)=T(x)\wedge |T(y)|=T(x)\wedge T(|y|)=T(x\wedge |y|).
              \end{equation*}
              By the solidness of $E$, we have $x\wedge |y|\in E$. 
              It follows that $T(x)\in T(E)$, showing that $T(E)$ is a solid subset of $G_2$.
  \item [(v)]If $T$ is bijective, then $T^{-1}$ is clearly a lattice homomorphism from $G_2$ to $G_1$.
             By (i), $T$ and $T^{-1}$ are both positive. Conversely, suppose $T$ and $T^{-1}$ are both positive. Since $x^+\geq 0$ and $x^+\geq x$ for any $x\in G_1$, we have $T(x^+)\geq 0$ and $T(x^+)\geq x$; hence $T(x^+)\geq (T(x))^+$. Apply this inequality to the map $T^{-1}$ and the element $T(x)\in G_2$ to obtain
             \begin{equation*}
             T^{-1}([T(x)]^+)\geq (T^{-1}(T(x)))^+=x^+,
             \end{equation*}
             which implies $(T(x))^+= T(x^+)$. It follows from Theorem \ref{theorem4.2} that $T$ is a lattice homomorphism
  \item [(vi)]Since $T$ is a group homomorphism, $Ker(T)$ is a subgroup of $G_1$. Next, we show $Ker(T)$
                is solid. To this end, assume $|x|\leq |y|, x\in G$ and $y\in Ker(T)$. By Theorem \ref{theorem4.2}, we have
                \begin{equation*}
                |T(x)|=T(|x|)=T(|x|\wedge |y|)=T(|x|)\wedge T(|y|)=T(|x|)\wedge 0=0,
                \end{equation*}
                implying $x\in Ker(T)$. Thus, $Ker(T)$ is a solid in $G_1$.
  \item [(vii)]If $T$ is $\sigma$-order-continuous, then (vi) implies that $Ker(T)$ is a $\sigma$-ideal
                of $G_1$. Conversely, assume $Ker(T)$ is a $\sigma$-ideal of $G_1$ and a sequence $x_n\downarrow 0$ in $G_1$. Since $T$ is positive by (i), it is easy to see that we only need to show $T(x_n)\downarrow 0$ in $G_2$.
                Clearly, the positivity of $T$ implies $T(x_n)\downarrow$; so it remains to show $\inf_{n}\{T(x_n)\}=0$.
               Suppose not. Then there exists $y\in G_2$ such that $0< y\leq T(x_n)$ for all $n\in N$.
                By Theorem \ref{theorem4.2}, we know there exists $x_0\in (G_1)_+$ such that $T(x_0)=y$. We have
                \begin{equation*}
                T\left((x_0-x_n)^+\right)=T\left((x_0)-(x_n)\right)^+=\left(y-T(x_n)\right)^+=0.
                \end{equation*}
                Thus, $x_0-x_n\in Ker(T)$ for all $n$. Since $0\leq (x_0-x_n)^+\uparrow x_0$, the order-closedness of $Ker(T)$ implies $x_0\in Ker(T)$, i.e., $T(x_0)=y=0$, contradicting $y>0$. Therefore, we must have $\inf_n\{T(x_n)\}=0$.
  \item [(viii)]Similar to (viii).
\end{enumerate}
\end{proof}

Let $G$ be an $l$-group and $H$ be a subgroup of $G$. Since $G$ is assumed to be commutative, $H$ is always is normal subgroup of $G$; hence the quotient group $G/H$ is well-defined. Following \cite{Fuchs1}, we order the quotient group $G/H$ as follows:
\begin{equation}\label{4.6}
\overline{x}\leq \overline{y} \text{ if and only if $a\leq b$,}
\end{equation}
where $a$ and $b$ are some representatives of $\overline{x}$ and $\overline{y}$, respectively. Then $G/H$ becomes a p.o. group. In the case where $H$ is an ideal, we can say more.

\begin{theorem}\label{theorem4.4}
If $A$ is an ideal of an $l$-group $G$, then the following three statements hold.
\begin{enumerate}
  \item [(i)]The the positive cone $(G/A)_+=\{\overline{x}\mid x\in G_+\}$ of $G/A$ satisfies the following three properties:
               \begin{enumerate}
                 \item [(1)]$(G/A)_+ +(G/A)_+\subset (G/A)_+$;
                 \item [(2)]$ n (G/A)_+\subset (G/A)_+$ for all positive integer $n$;
                 \item [(3)]$(G/A)_+\cap (-(G/A)_+)=\{0\}$.
               \end{enumerate}
  \item [(ii)]$G/A$ is an $l$-group.
  \item [(iii)]The natural projection $\pi: A\rightarrow G/A$ is an onto lattice homomorphism.
\end{enumerate}
\end{theorem}

\begin{proof}
\begin{enumerate}
  \item [(i)]Properties (1) and (2) are trivial. To see property (3), take $\overline{x}$ in $(G/A)_+\cap (-(G/A)_+)$. Then there exist positive elements $a$ and $b$ in $G$ such that $\overline{a}=\overline{(-b)}=\overline{x}$. Thus, $\overline{a+b}=\overline{0}$, implying $a+b\in A$.
              Since $0\leq x\leq a+b$ and $A$ is solid, we have $x\in A$. It follows that $\overline{x}=\overline{0}$; hence property (3) holds.
  \item [(ii)]We already know that if we order $G/A$ according to Equation (\ref{4.6}), then $G/A$ becomes
                a partially ordered group. So it suffices to show that $G/A$ is a lattice. Indeed, it suffices to show that $(\overline{x})^+$ exists in $G/A$ for each $\overline{x}\in G/A$ (cf. Theorem 8 of \cite{Birkhoff1} or p. 67 of \cite{Fuchs1}). Since $x\leq x^+$ and $0\leq x^+$ in $G$, Equation (\ref{4.1}) shows that $\overline{x}\leq \overline{x^+}$ and $0\leq \overline{x^+}$, that is, $\overline{x^+}$ is another upper bound of the set $\{\overline{0}, \overline{x}\}$. Next, suppose $\overline{y}$ is an upper bound of $\{\overline{0}, \overline{x}\}$, i.e, $\overline{y}\geq \overline{0}$ and $\overline{y}\geq \overline{x}$ in $G/A$. Take representatives $a$ and $b$ from $\overline{x}$ and $\overline{y}$, respectively.
                Then $a\leq b$. Without loss of generality, we may also assume $b\geq 0$. It follows that
                \begin{eqnarray*}
                x &=& a+(x-a)\leq b+(x-a)^+.
                \end{eqnarray*}
                Also, $0\leq b+(x-a)^+$. Thus, $x^+\leq b+(x-a)^+$,                 implying $\overline{x^+}\leq \overline{b}=\overline{y}$ in $G/A$. Therefore, $\overline{x^+}=\sup\{\overline{0}, \overline{x}\}=(\overline{x})^+$, proving that $(\overline{x})^+$ exists in $G/A$.
  \item [(iii)]By definition of the natural projection, $\pi$ is surjective. In the proof (ii), we have     obtained $\pi(x^+)=(\pi(x))^+$ for all $x\in G$. Therefore, Theorem \ref{theorem4.2} implies that $\pi$ is also a lattice homomorphism.
\end{enumerate}
\end{proof}
\noindent \textbf{Remark.} Properties (1) and (2) in statement (i) shows that $(G/A)_+$ is indeed a cone in the quotient group $G/A$.

\begin{corollary}\label{corollary4.1}
Suppose $(G_1, \tau_1)$ is a locally solid topological $l$-group, $G_2$ is an $l$-group, and $T$ is a lattice homomorphism from $G_1$ to $G_2$, then $(G_2, \tau_{T})$ is a locally solid topological $l$-group, where $\tau_{T}$ is the quotient topology on $G_2$ inducted by $T$. In particular, if $A$ is an ideal of a topological $l$-group $G$, then $(G/A, \tau_{\pi})$ is a locally solid topological $l$-group, where $\pi$ is the natural projection from $G$ to $G/A$.
\end{corollary}

\begin{proof}
Since $T$ is a group homomorphism from $G_1$ to $G_2$, the quotient topology $\tau_{T}$ is a group topology, making $(G_2, \tau_{T})$ into a topological $l$-group (cf. p. 59 of \cite{Husain}). Moreover, if we let $\mathcal{B}_0$ be a neighborhood base at zero consisting of solid sets, then $\{T(U)\mid U\in \mathcal{B}_0\}$ is a neighborhood base at zero for $\tau_{T}$. It follows from Theorem \ref{theorem4.3} that $\tau_T$ is locally solid, that is, $(G_2, \tau_{T})$ is a locally solid topological $l$-group.

If $A$ is an ideal of a topological $l$-group $G$, then Theorem \ref{theorem4.4} shows that the natural projection $\pi: G\rightarrow G/A$ is an onto lattice homomorphism. Therefore, the second statement follows immediately from the first statement.
\end{proof}

Next, we investigate order-bounded sets in a topological $l$-group. Recall that a subset $E$ of a topological group $(G, \tau)$ is said to be \emph{$\tau$-bounded} if for every $\tau$-neighborhood $U$ of zero there exists a positive integer $n$ such that $E\subset n U$. It is known that an order-bounded subset of a locally solid Riesz space is topologically bounded (cf. Theorem 2.19 of \cite{AB1}). However, this result does not extend to locally solid topological $l$-groups. Consider the following example.

\begin{example}\label{example4.2}
Let $G$ be the additive group of reals equipped with the usual order and discrete topology $\tau$. Then $(G, \tau)$ is evidently a locally solid topological $l$-group.
Since a Riesz space is connected, $G$ is not a Riesz space. Let $U=B(0, 1)$ be the open ball centered at $0$ with radius $1$. Then $U$ is a neighborhood of zero and $U=\{0\}$.
Choose $E=[-2014, 2014]$. Then $E$ is clearly an order-bounded subset of $G$. However, for all positive integer $n$ we have $E\not \subset nU$; hence $E$ is not $\tau$-bounded.
\end{example}
\noindent \textbf{Remark.} Indeed, the fact that an order-bounded subset of a locally solid Riesz space $(L, \tau)$ is $\tau$-bounded depends on the fact that each neighborhood of zero is absorbing which in turn depends on the continuity of the scalar multiplication. Since a topological $l$-group lacks this property, an order-bounded set in a locally solid topological $l$-group is not expected to be topologically bounded. \\

The next theorem give a condition under which a $\tau$-bounded subset of a topological $l$-group will be order-bounded.
\begin{theorem}\label{theorem4.5}
Let $(G, \tau)$ be a topological $l$-group. If $G$ has an order-bounded $\tau$-neighborhood of zero, then every $\tau$-bounded subset is order-bounded.
\end{theorem}

\begin{proof}
Let $\mathcal{B}_0$ be a $\tau$-neighborhood base of zero. By hypothesis, there exists $U\in \mathcal{B}_0$ such that $U$ is contained in some order interval $[x, y]$ of $G$, where $x, y\in G$. Suppose $E$ is a $\tau$-bounded subset of $G$. Then there exists a positive integer $n$ such that $E\subset n U$. It follows from the hypothesis that $E$ is contained in the order interval $[nx, ny]$ of $G$, showing that $E$ is order-bounded.
\end{proof}

The next result shows that order-bounded sets in a topologically group have some desirable properties.
\begin{theorem}\label{theorem4.6}
Suppose $(G, \tau)$ is a topological $l$-group. Then the following statements hold.
\begin{enumerate}
  \item [(i)]An arbitrary intersection of ordered bounded sets is order-bounded.
  \item [(ii)]A finite union of ordered bounded sets is order-bounded.
  \item [(iii)]The algebraic sum of two ordered bounded sets is order-bounded.
  \item [(iv)]A nonzero multiple of an order-bounded set is order-bounded.
   \item [(v)]If $A$ is an ideal in $L$ and $\pi :L\rightarrow L/A$ is the natural projection, then
              $\pi$ maps an order-bounded set to an order-bounded set, i.e., $\pi$ is an order-bounded homomorphism.
\end{enumerate}
\end{theorem}
\begin{proof}
(i)-(iv) are trivial. We show (v). Since $A$ is an ideal of $L$, Theorem \ref{theorem4.4} (iii) shows that the natural projection $\pi: L\rightarrow L/A$ is a lattice homomorphism. Thus, $T$ is a positive homomorphism. Then the conclusion follows from the fact that every positive homomorphism between two $l$-group is order-bounded.
\end{proof}

The next theorem gives more properties of locally solid topological $l$-groups.
\begin{theorem}\label{theorem4.7}
Suppose $(G, \tau)$ is a locally solid topological $l$-group. Then the following two statements hold.
\begin{enumerate}
  \item [(i)]The $\tau$-closure of an $l$-subgroup of $G$ is an $l$-group.
  \item [(ii)]The $\tau$-closure of a solid subset of $G$ is solid.
  \item [(iii)]The $\tau$-closure of an ideal in $G$ is an ideal.
\end{enumerate}
\end{theorem}

\begin{proof}
\begin{enumerate}
  \item [(i)]Let $H$ be an $l$-subgroup of $G$. By Theorem \ref{theorem3.1} (iii), the closure $\overline{H}$ of $H$ is subgroup of $G$. Let $x_0\in \overline{H}$. Then there exists a net $(x_{\alpha})$ in $H$ such that $x_{\alpha}\xrightarrow{\tau}x_0$. Since $H$ is an $l$-subgroup of $G$, the net $(x^+_{\alpha})$ belongs to $H$. By the continuity of the map $x\mapsto x^+$, we have $x_{\alpha}^+\xrightarrow{\tau}x^+_0$, implying $x^+_0\in H$.
      Therefore, $\overline{H}$ is an $l$-group in view of Theorem 8 of \cite{Birkhoff1}.
  \item [(ii)]Let $E$ be a solid subset of $G$. Suppose $|x|\leq |y|$ in $G$ and $y\in \overline{E}$. Then there exists a net $(y_{\alpha})$ in $G$ such that $y_{\alpha}\xrightarrow{\tau}y$. Define a two-sided truncated net $(z_{\alpha})$ as follows:
                \begin{equation*}
                z_{\alpha}=\left\{
                             \begin{array}{ll}
                               x\wedge |y_{\alpha}|, & \hbox{if $x\geq 0$;} \\
                               (-x)\vee (-|y_{\alpha}|), & \hbox{if $x<0$.}
                             \end{array}
                           \right.
                \end{equation*}
                Then the solidness of $E$ implies that the net $(z_{\alpha})$ belong to $H$. In addition, Theorem \ref{theorem4.1} shows $z_{\alpha}\xrightarrow{\tau}x$; hence $x\in \overline{E}$. This proves that $\overline{E}$ is a solid subset of $G$.
  \item [(iii)]This follows from (i) and (ii).
\end{enumerate}
\end{proof}

We close this section by giving some properties of Hausdorff locally solid topological $l$-groups.
\begin{theorem}\label{theorem4.8}
Suppose $(G, \tau)$ is a Hausdorff locally solid topological $l$-group. Then the following statement hold.
\begin{enumerate}
  \item [(i)]The positive cone $G_+$ is $\tau$-closed.
  \item [(ii)]Let $(x_{\alpha})_{\alpha\in A}$ be a net in $G$. If $x_{\alpha}\xrightarrow{\tau} x$ and $x_{\alpha}\downarrow $ in $G$, then $x_{\alpha}\downarrow x$. Likewise, if $x_{\alpha}\xrightarrow{\tau} x$ and $x_{\alpha}\uparrow $ in $G$, then $x_{\alpha}\uparrow x$.
  \item [(iii)]Let $(x_{\alpha})_{\alpha\in A}$ and $(y_{\alpha})_{\alpha\in A}$ are two nets in $G$. If $x_{\alpha}\leq y_{\alpha}\downarrow$ and $x_{\alpha}-x_{\alpha}\xrightarrow{\tau} 0$, then $x_{\alpha}\downarrow x$ if and only if $y_{\alpha}\downarrow x$.
\item [(iv)] If $\{x_{\alpha}\}$ is an increasing net in $G$ with a cluster point $x_0$, then
            $x_{\alpha}\uparrow x_0$.
\item [(v)]If $E$ is a subset of $G$ and $x\in \overline{E}$, then
            $x=\sup\{x\wedge y\mid y\in G\}=\inf\{x\vee y\mid y\in G\}$.
\end{enumerate}
\end{theorem}

\begin{proof}
\begin{enumerate}
  \item [(i)]Theorem \ref{theorem3.3} shows that $\{0\}$ is $\tau$-closed. Since the positive cone $G_+$ can be written as $G_+=\{x\mid x^-=0\}$, the conclusion follows the continuity of the map $x\mapsto x^-$.
  \item [(ii)]Fix an index $\alpha_0$. Since the net $(x_{\alpha})_{\alpha\in A}$ is decreasing, for any $\alpha\geq \alpha_0$ we have
             \begin{equation*}
             0\leq x-x_{\alpha_0}\wedge x\leq x-x_{\alpha}\wedge x\leq |x-x_{\alpha}|.
             \end{equation*}
             It follows that $x-x_{\alpha_0}\wedge x=0$, implying $x\leq x_{\alpha}$ for all $\alpha\in A$. This shows that $x$ is a lower bound of $\{x_{\alpha}\}_{\alpha\in A}$. Next, suppose $y \in G$ is another upper bound of $\{x_{\alpha}\}_{\alpha\in A}$, i.e.,  there exists a $y\in G$ such that $y\leq x_{\alpha}$ for all $\alpha\in A$. By hypothesis, we have
             \begin{equation*}
             0\leq x_{\alpha}-y\xrightarrow{\tau}x-y.
             \end{equation*}
             It follows from (i) that $x-y\in G_+$, i.e., $y\leq x$. Therefore, $x=\inf_{\alpha\in A}\{x_{\alpha}\}$. This shows that $x_{\alpha}\downarrow x$ in $G$.
  \item [(iii)]First, we assume $y_{\alpha}\downarrow x$. Then the hypothesis implies
              \begin{equation*}
              0\leq x-x\wedge x_{\alpha}\leq y_{\alpha}-x_{\alpha}\xrightarrow{\tau}0.
              \end{equation*}
              Hence, we have $x-x\wedge x_{\alpha}\xrightarrow{\tau} 0$.
              It follows from (ii) that $0\leq x-x\wedge x_{\alpha}\uparrow 0$, yielding
              $x-x\wedge x_{\alpha}=0$. Thus, $x\leq x_{\alpha}$ for all $\alpha\in A$.
              Therefore, we have $x\leq x_{\alpha}\leq y_{\alpha}\downarrow$. Since $y_{\alpha}\leq x$, we must have $x_{\alpha}\downarrow x$.

              Next, we assume $x_{\alpha}\downarrow x$. Suppose there exists some $y\in G$ such that $x\leq y\leq y_{\alpha}$ for all $\alpha\in A$. Then for all $\alpha\in A$ we have
              \begin{equation*}
              0\leq (y-x_{\alpha})^+\leq (y_{\alpha}-x_{\alpha}).
              \end{equation*}
              By hypothesis, $(y-x_{\alpha})^+\xrightarrow{\tau} 0$. Since the net $(x_{\alpha})_{\alpha\in A}$ is decreasing, we have
              \begin{equation*}
              (y-x_{\alpha})^+\uparrow (y-x)^+=y-x.
              \end{equation*}
              It follows from (ii) that $x=y$. This shows that $y_{\alpha}\downarrow x$.
  \item [(iv)]Since $x_0$ is a cluster point of $\{x_{\alpha}\}$, there exists an increasing subnet $\{x_{\alpha_{\beta}}\}$
              of $\{x_{\alpha}\}$ such that $x_{\alpha_{\beta}}\xrightarrow{\tau}x_0$.
              It follows from (ii) that $x_{\alpha_{\beta}}\uparrow x_0$, that is, $\sup\{x_{\alpha_{\beta}}\}=x_0$ Since $\{x_{\alpha}\}$ is increasing, for each $\alpha$
              we may choose a $\beta_0$ such that $x_{\alpha_{\beta}} -  x_{\alpha}\geq 0$ for all $\beta\geq \beta_0$.
              Since $x_{\alpha_{\beta}}-x_{\alpha}\xrightarrow{\tau} x_0-x_{\alpha}$, (i) implies $x_0-x_{\alpha}\in G_+$, i.e.,
              $x_0\geq x_{\alpha}$ for all $\alpha$; hence we have $\sup \{x_{\alpha}\}\leq x_0$. Clearly, $x_0=\sup\{x_{\alpha_{\beta}}\}
              \leq \sup\{x_{\alpha}\}$; hence we must have $\sup\{x_{\alpha}\}=x_0$. Therefore, $x_{\alpha}\uparrow x_0$.
  \item [(v)]We prove the first equality only as the second can be proved in a similar manner. It is evident that
             $x$ is an upper bound of the set $\{x\wedge y\mid y\in G\}$. Choose a net $\{x_{\alpha}\}$ in $G$ such that
             $x_{\alpha}\xrightarrow{\tau}x$.  If $z$ is another upper bound of
             $\{x\wedge y\mid y\in G\}$, then we have $z-x\wedge x_{\alpha}\geq 0$ for all $\alpha$.
             Since $x\wedge x_{\alpha}\xrightarrow{\tau} x$, (i) shows that $z-x\geq 0$, i.e. $z\geq x$. Therefore,
             $x=\sup \{x\wedge y\mid y\in G\}$.
\end{enumerate}
\end{proof}

\section{Topological completion of Hausdorff locally solid $l$-groups}
Every topological group induces a uniform space; thus the concept of completeness is well-defined. Since we assume all groups are commutative, every topological group $(G, \tau)$ has a completion $(\widehat{G}, \widehat{\tau})$, though the completion may not be unique. In this section, we further assume that every topological group is Hausdorff. Then we know the completion $(\widehat{G}, \widehat{\tau})$ of $(G, \tau)$ is unique (up to group isomorphism) and Hausdorff (cf. p. 6 of \cite{BNS}). Specifically, the following theorem holds.

\begin{theorem}\label{theorem5.1}
If $(G, \tau)$ is a Hausdorff topological group, then there exists a unique (up to group isomorphism) Hausdorff topological group $(\widehat{G}, \widehat{\tau})$ having the following properties:
\begin{enumerate}
  \item [(i)]The Hausdorff topological group $(\widehat{G}, \widehat{\tau})$ is complete.
  \item [(ii)]There exists a subgroup $H$ of $\widehat{G}$ such that $H$ is isomorphic to $\widehat{G}$; hence $G$ is identified as a subgroup of $\widehat{G}$.
  \item [(iii)]The topology $\widehat{\tau}$ induces $\tau$ on $G$.
  \item [(iv)]The subgroup $G$ is $\tau$-dense in $\widehat{G}$.
  \item [(v)]If the subgroup $G$ is an ideal of $\widehat{G}$, then $G$ is order dense in $\widehat{G}$.
\end{enumerate}
In particular, if $\mathcal{B}_0$ is a $\tau$-neighborhood base at zero, then $\overline{\mathcal{B}}_0=\{\overline{U} \mid U\in \mathcal{B}_0\}$ is a $\widehat{\tau}$-neighborhood base at zero. We say $(\widehat{G}, \widehat{\tau})$ is a topological completion of $(G, \tau)$.
\end{theorem}

\begin{proof}
Only (v) needs a proof. If $G$ is an ideal of $\widehat{G}$, then for every $0<\widehat{x}\in \widehat{G}$ there exists a net $\{x_{\alpha}\}$ in $G$ such that $x_{\alpha}\xrightarrow {\widehat{\tau}} \widehat{x}$. Without loss of generality, we may assume $x_{\alpha}\neq 0$ for all $\alpha$. Clearly, for each $\alpha$ we have
\begin{equation*}
0<x_{\alpha}\wedge \widehat{x}\leq x_{\alpha}\in G.
\end{equation*}
Therefore, each $x_{\alpha}\wedge \widehat{x}$ belongs to $G$, showing that $G$ is order-dense in $\widehat{G}$.
\end{proof}
\noindent \textbf{Remark.} Indeed, the above proof also shows that if $G$ is an ideal of $\widehat{G}$, then there exists a positive increasing net $\{x_{\alpha}\}$ in $G$ such that $x_{\alpha} \xrightarrow{\widehat{\tau}}\widehat{x}$.\\

It is natural to ask whether $(G, \tau)$ is an $l$-subgroup of $(\widehat{G}, \widehat{\tau})$ if in addition $G$ is an $l$-group and $\tau$ is locally solid topology. This analogous problem for locally solid Riesz spaces was studied by several author (cf. \cite{Aliprantis1}, \cite{Fremlin2} and \cite{Kawai}) and an affirmative answer was given. The next theorem shows that we also have an affirmative answer in the case of topological $l$-groups.

\begin{theorem}\label{theorem5.2}
Suppose $(G, \tau)$ is a Hausdorff topological $l$-group and $(\widehat{G}, \widehat{\tau})$ is its topological completion. Then the $\widehat{\tau}$-closure $\overline{G_+}$ of $G_+$ is a cone of $\widehat{G}$ and $(\widehat{G}, \widehat{\tau})$ equipped with the partial order induced by $\overline{G_+}$ is a Hausdorff locally solid $l$-group containing $G$ as a $l$-subgroup. In addition, the $\widehat{\tau}$-closure of a solid subset of $G$ is a solid subset of $\widehat{G}$.
\end{theorem}

\begin{proof}
First, we show that the $\widehat{\tau}$-closure $\overline{G_+}$ of $G_+$ is a cone in $\widehat{G}$. To this end, we need to verify conditions (i) and (ii) in the definition of cones (cf. Section 2).
To verify (i), notice that
$G_+ + G_+\subset G_+$ holds trivially; hence the continuity of addition immediately leads to (i). To verify (ii), take $x\in \overline{G_+}\cap (-\overline{G_+})$. Then there exists two nets $\{x_{\alpha}\}$ and $\{y_{\beta}\}$ in $G$ such that $x_{\alpha}\xrightarrow{\widehat{\tau}}x$ and $y_{\beta}\xrightarrow{\widehat{\tau}} -x$. Thus, $0\leq x_{\alpha}\leq x_{\alpha}+y_{\beta} \underset {(\alpha, \beta)} {\xrightarrow{\tau}} 0 $. It follows that $x_{\alpha}\xrightarrow{\tau} 0$; hence $x=0$. This shows that $x\in \overline{G_+}\cap (-\overline{G_+})=\{0\}$. Therefore, $\overline{G_+}$ is a cone of $\widehat{G}$.

Let $G$ be ordered by the partial order induced by $\widehat{G}_+$ according to Equation (\ref{4.6}). To complete the proof of the first statement, it suffices to show that $(\widehat{x})^+$ exists in $\widehat{G}$ for all $x\in \widehat{G}$ (cf. Theorem 8 in \cite{Birkhoff1}). Since $\tau$ is locally solid, Theorem \ref{theorem4.1} shows that the map $T: x\rightarrow x^+$ is uniformly continuous. In view of Theorem \ref{theorem5.1}, there exists a unique uniformly continuous extension $\overline{T}$ of $T$ to $(\widehat{G}, \widehat{\tau})$. Thus, it remains to show that $\overline{T}(\widehat{x})=(\widehat{x})^+$. To this end, take $\widehat{x}\in \widehat{G}$ and choose a net $\{x_{\alpha}\}$ in $G$ such that $x_{\alpha}\xrightarrow{\widehat{\tau}}\widehat{x}$.
Since $\overline{T}(x_{\alpha})=x^+_{\alpha}\geq x_{\alpha}$, i.e., $x_{\alpha}^+-x_{\alpha}\in G_+$, we have
$T(\widehat{x})-\widehat{x}=(x)^+ - \widehat{x}\in \widehat{G}$, showing $\overline{T}(\widehat{x})\geq (\widehat{x})^+$. If $\widehat{y}\in (\widehat{G})_+$ and $\widehat{y}>(\widehat{x})^+$, then $\widehat{y}>0$; hence $\overline{T}(\widehat{y})=\widehat{y}$. Choose another net $\{y_{\alpha}\}$ with the same index set as $\{x_{\alpha}\}$ such that $y_{\alpha}\xrightarrow{\widehat{\tau}}\widehat{y}$; then
\begin{equation*}
y_{\alpha}-x_{\alpha} \xrightarrow{\widehat{\tau}} \widehat{y}-\widehat{x}\geq \widehat{y} -(\widehat{x})^+>0,
\end{equation*}
implying $\widehat{y}-\overline{T}(\widehat{x})>0$, i.e, $\widehat{y}>\overline{T}(\widehat{x})$. Therefor, we must have $\widehat{x}=\overline{T}(\widehat{x})$.

To show the second statement, we let $E$ be a solid subset in $G$ and $\overline{E}$ be its $\widehat{\tau}$-closure in $\widehat{G}$.
Suppose $|\widehat{x}|\leq |\widehat{y}|$, where $\widehat{x}\in \widehat{G}$ and $\widehat{y}\in \overline{E}$.
Then we can choose a net $\{x_{\alpha}\}$ in $\overline{E}$ such that $x_{\alpha}\xrightarrow{\widehat{\tau}}x$ and a net $\{y_{\beta}\}$ in $G$ such that $y_{\beta}\xrightarrow {\widehat{\tau}}y$. Define a net $\{z_{\alpha, \beta}\}$ via the formula
\begin{equation*}
z_{\alpha, \beta}=\left\{
                    \begin{array}{ll}
                      x_{\alpha}\wedge |y_{\beta}|, & \hbox{if $x_{\alpha}\geq 0$;} \\
                      (-x_{\alpha})\vee (-|y_{\beta}|), & \hbox{if $x_{\alpha}<0$.}
                    \end{array}
                  \right.
\end{equation*}
Then $z_{\alpha, \beta}\xrightarrow{\widehat{\tau}} \widehat{x}$. Also, $|z_{\alpha, \beta}|\leq |y_{\beta}|$
for all $\alpha$ and $\beta$; hence the net $\{z_{\alpha, \beta}\}$ is in $E$. 
It follows that $x\in \overline{E}$, proving that $\overline{E}$ is solid.
\end{proof}

\begin{corollary}\label{corollary5.1}
Suppose $(G, \tau)$ is a Hausdorff locally solid topological $l$-group and $(\widehat{G}, \widehat{\tau})$ is its topological completion. Then $G$ is an $l$-subgroup of $\widehat{G}$. In addition, if $\mathcal{B}_0$ is a $\tau$-neighborhood base at zero, then $\overline{\mathcal{B}_0}=\{\overline{U}\mid U\in \mathcal{B}_0\}$ is a $\widehat{\tau}$-neighborhood base at zero consisting of solid sets.
\end{corollary}

\begin{theorem}\label{theorem5.3}
Let $(\widehat{G}, \widehat{\tau})$ be the topological completion of a Hausdorff locally solid topological $l$-group $(G, \tau)$. Then the following two statements are equivalent.
\begin{enumerate}
  \item [(i)]$x_{\alpha}\downarrow 0$ in $G$ implies $x_{\alpha}\downarrow 0$ in $\widehat{G}$ for all net $\{x_{\alpha}\}$ in $G$.
  \item [(ii)]If $\{x_{\alpha}\}$ is a positive Cauchy $\tau$-net and $x_{\alpha}\xrightarrow{o} 0$, then
                $x_{\alpha}\xrightarrow{\tau} 0$.
\end{enumerate}
\end{theorem}

\begin{proof}
(i) $\Longrightarrow$ (ii). Let $\{x_{\alpha}\}$ be a positive $\tau$-Cauchy net with $x_{\alpha}\xrightarrow{o} 0$. Then there exists $\widehat{x}\in \widehat{G}$ such that $x\xrightarrow{\widehat{\tau}}\widehat{x}$. We need to show $\widehat{x}=0$ which implies $x_{\alpha}\xrightarrow{\tau} 0$. By the definition of order-convergence, there exists a net $\{ y_{\alpha} \}$ in $G$ such that $0\leq x_{\alpha}\leq y_{\alpha}\downarrow 0$ in $G$. Fix an index $\beta$, we have
\begin{equation*}
0\leq x_{\alpha}\leq y_{\alpha}\leq y_{\beta}, \quad\text{for $\alpha\geq\beta$}.
\end{equation*}
Therefore, we have $0\leq \widehat{x}\leq y_{\beta}\downarrow 0$ in $G$. By hypothesis, we have $0\leq x\leq y_{\beta}\downarrow 0$ in $\widehat{G}$, implying $\widehat{x}=0$.

(ii) $\Longrightarrow$ (i). Let $x_{\alpha}\downarrow 0$ in $G$. We need to show $x_{\alpha}\downarrow 0$ in $\widehat{G}$. Suppose not. Then there exists $\widehat{y}$ in $\widehat{G}$ such that $0<\widehat{y}\leq x_{\alpha}$ in $\widehat{G}$ for all $\alpha$. Choose a net $\{y_{\beta}\}$ in $G_+$ such that $y_{\beta}\xrightarrow{\widehat{\tau}} \widehat{y}$ and invoke Theorem 19 of \cite{Birkhoff1} to obtain
\begin{equation*}
|x_{\alpha} \wedge y_{\beta}-\widehat{y}|\leq |y_{\beta}-\widehat{y}|\ \text{for all $\alpha$ and $\beta$}.
\end{equation*}
It follows that $x_{\alpha}\wedge y_{\beta}\underset{(\alpha, \beta)}{\xrightarrow{\widehat{\tau}}} \widehat{y}>0$. On the other hand, $0\leq x_{\alpha}\wedge y_{\beta}\leq x_{\alpha}$ and $x_{\alpha}\downarrow 0$ imply $x_{\alpha}\wedge y_{\beta} \underset{(\alpha, \beta)}{\xrightarrow{o}} 0$. By hypothesis, we have
$x_{\alpha}\wedge y_{\beta} \underset{(\alpha, \beta)}{\xrightarrow{\widehat{\tau}}}0$, a contradiction.
\end{proof}

\begin{theorem}\label{theorem5.4}
Suppose $(G, \tau)$ is a Hausdorff locally solid topological $l$-group and $(\widehat{G}, \widehat{\tau})$ is its topological completion. Then the following two statements are equivalent.
\begin{enumerate}
  \item [(i)]$G$ is an ideal of $\widehat{G}$.
  \item [(ii)]Every order interval of $G$ is $\tau$-complete.
\end{enumerate}
\end{theorem}

\begin{proof}
(i) $\Longrightarrow$ (ii). Let $[x, y]$ be an order interval in $G$ and $\{z_{\alpha}\}$ be a $\tau$-Cauchy net in $[x, y]$. Then $x\leq z_{\alpha}\leq y$ for all $\alpha$. Since $z_{\alpha}\xrightarrow {\widehat{\tau}} \widehat{z}$ for some $\widehat{z}\in \widehat{G}$, we have $x\leq \widehat{z} \leq y$. If $\widehat{z}\geq 0$, then $y\in G_+$; hence $\widehat{z}\in G$ by the solidness of $G$. If $\widehat{z}<0$, then $|\widehat{z}|\leq |x|$ and $x\in G$; the solidness of $G$ implies $\widehat{z}\in G$. In either case, we have $\widehat{z}\in G$, showing that $[x, y]$ is $\tau$-complete.

(ii) $\Longrightarrow$ (i). Assume $\widehat{x}\leq y$, where $\widehat{x} \in \widehat{G}$ and $y\in G$. Choose a net $\{x_{\alpha}\}$ in $G$ such that $x_{\alpha}\xrightarrow{\widehat{\tau}} \widehat{x}$. If $\widehat{x}\geq 0$, then $\widehat{x}\in [0, |y|]\subset G$; if $\widehat{x}<0$, then $\widehat{x}\in [-|y|, 0]\subset G$. In either case, we have $\widehat{x}\in G$, showing that $G$ is an ideal of $\widehat{G}$.
\end{proof}

Theorem \ref{theorem5.1} (v) and Theorem \ref{theorem5.4} immediately lead to the following theorem.

\begin{theorem}\label{theorem5.5}
Suppose $(G, \tau)$ is a Hausdorff locally solid topological $l$-group. Then the following two statements are equivalent.
\begin{enumerate}
  \item [(i)]$(G, \tau)$ is $\tau$-complete.
  \item [(ii)]Every order interval of $G$ is $\tau$-complete and each increasing $\tau$-Cauchy net in $G_+$ is $\tau$-convergent.
\end{enumerate}
\end{theorem}

The sequential version of the above theorem is as follows.

\begin{theorem}\label{theorem5.6}
Suppose $(G, \tau)$ is a Hausdorff locally solid topological $l$-group. Then the following two statements are equivalent.
\begin{enumerate}
  \item [(i)]$(G, \tau)$ is sequentially $\tau$-complete.
  \item [(ii)]Every order interval of $G$ is sequentially $\tau$-complete and each increasing $\tau$-Cauchy sequence in $G_+$ is $\tau$-convergent.
\end{enumerate}
\end{theorem}

\begin{proof}
(i) $\Longrightarrow$ (ii) is trivial. We show (ii) $\Longrightarrow$ (i). Let $(\widehat{G}, \widehat{\tau})$ be the unique Hausdorff toplogical completion of $(G, \tau)$. We need to show $G=\widehat{G}$. To this end, take any $\widehat{x}\in \widehat{G}$. Without loss of generality, we may assume $\widehat{x}\geq 0$.
Choose a sequence $\{x_n\}$ in $G_+$ such that $x_n\xrightarrow {\widehat{\tau}}\widehat{x}$. Then the truncated sequence $\{\widehat{x}\wedge x_n\}$ satisfies $ \widehat{x}\wedge x_n \xrightarrow{\widehat{\tau}} \widehat{x}$. Hence, the sequence $\{\widehat{x}\wedge x_n\}$ is a $\tau$-Cauchy sequence; it is also clear that this sequence is contained in the order interval $[0, \widehat{x}]$. Next, we form a new sequence $\{y_n\}$ by taking finite suprema of $\{\widehat{x}\wedge x_n\}$, that is, for each $n$ we define
\begin{equation*}
y_n=\sup_{1\leq k\leq n} \widehat{x}\wedge x_k.
\end{equation*}
Then $\{y_n\}$ is an increasing $\tau$-Cauchy sequence and $y_n\in [0, \widehat{x}]$ for all $n$.
Since $y_n\xrightarrow{\widehat{\tau}} \widehat{x}$, it follows from the hypothesis that
$\widehat{x}\in G$, proving that $\widehat{G}\subset G$. Therefore, $G=\widehat{G}$.
\end{proof}

\section*{Acknowledgments}

\noindent Thanks are due to the anonymous referee for comments and suggestions which led to several improvements in this article.

\bibliographystyle{amsplain}

\end{document}